\documentclass[10pt]{amsart}
\usepackage{amssymb}
\usepackage{amsmath}
\usepackage{amsthm}
\usepackage{latexsym}
\usepackage{graphicx}
\usepackage{hyperref}

\newcommand{\Z}{{\mathbb Z}}
\newcommand{\R}{{\mathbb R}}
\newcommand{\C}{{\mathbb C}}
\newcommand{\D}{{\mathbb D}}
\newcommand{\T}{{\mathbb T}}
\newcommand{\N}{{\mathbb N}}

\newtheorem{lemma}{Lemma}[section]
\newtheorem{theorem}[lemma]{Theorem}

\newtheorem{remark}[lemma]{Remark}
\newtheorem{proposition}[lemma]{Proposition}
\newtheorem{corollary}[lemma]{Corollary}
\newtheorem{definition}[lemma]{Definition}

\newcommand{\nn}{\nonumber}
\newcommand{\be}{\begin{equation}}
\newcommand{\ee}{\end{equation}}
\newcommand{\ul}{\underline}
\newcommand{\ol}{\overline}
\newcommand{\ti}{\tilde}

\newcommand{\spr}[2]{\left\langle #1 , #2 \right\rangle}

\newcommand{\E}{\mathrm{e}}
\newcommand{\I}{\mathrm{i}}

\newcommand{\tr}{\mathrm{tr}}

\DeclareMathOperator{\dist}{dist}

\newcommand{\eps}{\varepsilon}

\numberwithin{equation}{section}

\begin{document}

\title[Verblunsky coefficients defined by the skew-shift]{Orthogonal polynomials on the unit circle with Verblunsky coefficients defined by the skew-shift}

\author[H.\ Kr\"uger]{Helge Kr\"uger}
\address{Mathematics 253-37, Caltech, Pasadena, CA 91125}
\email{\href{helge@caltech.edu}{helge@caltech.edu}}
\urladdr{\href{http://www.its.caltech.edu/~helge/}{http://www.its.caltech.edu/~helge/}}

\thanks{H.\ K.\ was supported by a fellowship of the Simons foundation.}

\date{\today}


\begin{abstract}
 I give an example of a family of 
 orthogonal polynomials on the unit circle
 with Verblunsky coefficients
 given by the skew-shift for which the associated
 measures are supported on the entire unit circle
 and almost-every Aleksandrov measure is pure point.

 Furthermore, I show in the case of the two dimensional
 skew-shift the zeros of para-orthogonal polynomials
 obey the same statistics as an appropriate irrational
 rotation.

 The proof is based on an analysis of the associated
 CMV matrices.
\end{abstract}

\maketitle

%
%
%

\section{Introduction}

In this article, I consider orthogonal polynomials on
the unit circle, whose Verblunsky coefficients are given
by
\be\label{eq:defVeromegank}
 \alpha_n = \lambda \E^{2\pi\I\cdot \omega n^{k}}
\ee
for $0\neq \lambda\in\D=\{z:\quad |z|<1\}$,
$\omega$ an irrational number, and $k\geq 2$.
The case $k=1$ corresponds to rotated versions
of the Geronimus polynomials, see Theorem~1.6.13
in \cite{opuc2} and Proposition~\ref{prop:rotatealpha}
(see also Theorem~5.3 in \cite{gzborg}).
Given Verblunsky coefficients $\alpha_n$,
we define orthogonal polynomials
on the unit circle recursively by
\be\label{eq:defPhin}
 \Phi_0(z) = 1, \quad \Phi_{n+1}(z) = z \Phi_n(z)
 -\ol{\alpha_n} \Phi_n^{\ast}(z),
\ee
where $\Phi_n^{\ast}(z) = z^n \ol{\Phi(\ol{z}^{-1})}$
is the reversed polynomial. By Verblunsky's theorem,
there exists an unique probability measure $\mu$ on $\partial\D$ such that
the $\Phi_n$ are orthogonal with respect to it.
The first result is

\begin{theorem}\label{thm:main1}
 The support of $\mu$ satisfies
 \be
  \mathrm{supp}(\mu) = \partial\D.
 \ee
\end{theorem}

The key to the proof of this theorem is that the
support of $\mu$ is the same as the support
of the measure with Verblunsky coefficients
$\alpha_n \E^{2\pi\I y n}$ by ergodicity for
any $y\in\T=\R/\Z$. Now these two supports are
just rotated versions of each other. Hence
$\mathrm{supp}(\mu)$ must be the entire unit circle.
I give the details of the proof in Section~\ref{sec:pfmain1}.

Next, consider the family of Verblunsky coefficients
given by $\alpha_{x,n} = \alpha_n \cdot \E^{2\pi\I x}$.
The corresponding measures are known as {\em Aleksandrov
measures} $\mu_x$ see Section~3.2. in \cite{opuc1}.
Then we have that

\begin{theorem}\label{thm:main2}
 For almost every $x$, the Aleksandrov measure $\mu_x$
 is pure point.
\end{theorem}

The proof of this theorem is essentially the same as 
Theorem~\ref{thm:main1}, since the rotational invariance implies
positivity of the Lyapunov exponent. Pure point spectrum
then follows from spectral averaging.
Deterministic examples with similar properties
have been previously obtained in \cite{dk}.

Adapting the methods of \cite{krho1}, \cite{krho2}
to orthogonal polynomials on the unit circle, it should
be possible to obtain similar even for $k  > 1$
not an integer.

\bigskip

At this point, let me mention that the corresponding
question for orthogonal polynomials on the real line
respectively better Schr\"odinger operators is open.
Consider the potential $V(n) = \lambda \cos(2\pi \omega n^2)$
for an irrational number $\omega$. Then under a
Diophantine assumption on $\omega$ and a largeness
condition on $\lambda$ one can show pure point
spectrum, see \cite{b2002}, \cite{bgs},
and Chapter~15 in \cite{bbook} and that
the spectrum contains intervals \cite{kskew}.
However, it is believed that for all $\lambda > 0$
the spectrum of this operator is an interval
and pure point. Partial results for $\lambda > 0$
small can be found in \cite{b, b2,b3}.

The proofs of Theorem~\ref{thm:main1} and \ref{thm:main2}
are much easier than the real case, because of algebraic
miracles (Proposition~\ref{prop:rotatealphas}). However, there is also an analytic reason
why the case on the unit circle should be simpler,
namely that then the spectrum has no edges.

For this reason, I expect it to be possible to show
analogs of Theorem~\ref{thm:main1} and \ref{thm:main2}
if one perturbs $\alpha_n$ slightly by for example
$\alpha_n + \eps f(\omega n^k)$ for an analytic
and one-periodic function $f$ and $\eps > 0$
small enough.

\medskip

At first sight Theorem~\ref{thm:main1} and \ref{thm:main2}
might not seem too surprising, since we know many measures
whose support is the entire unit circle. But the Verblunsky
coefficients of these measures behave quite differently,
for regular measures one knows \cite{si} that the
Verblunsky coefficients Ces\'aro sum to $0$. Similarly
non-zero periodic potentials have at least one gap.

The situation becomes even more striking when considering
Schr\"odinger operators. There have been a series of 
innovative works \cite{abd1,abd2,aj1,aj2,gs4,gsfest} to prove
Cantor spectrum, whereas there are only the perturbative
methods from \cite{cs,kskew} to prove that the spectrum
contains an interval.

\bigskip

\begin{figure}[ht]
 \includegraphics[width=0.9\textwidth]{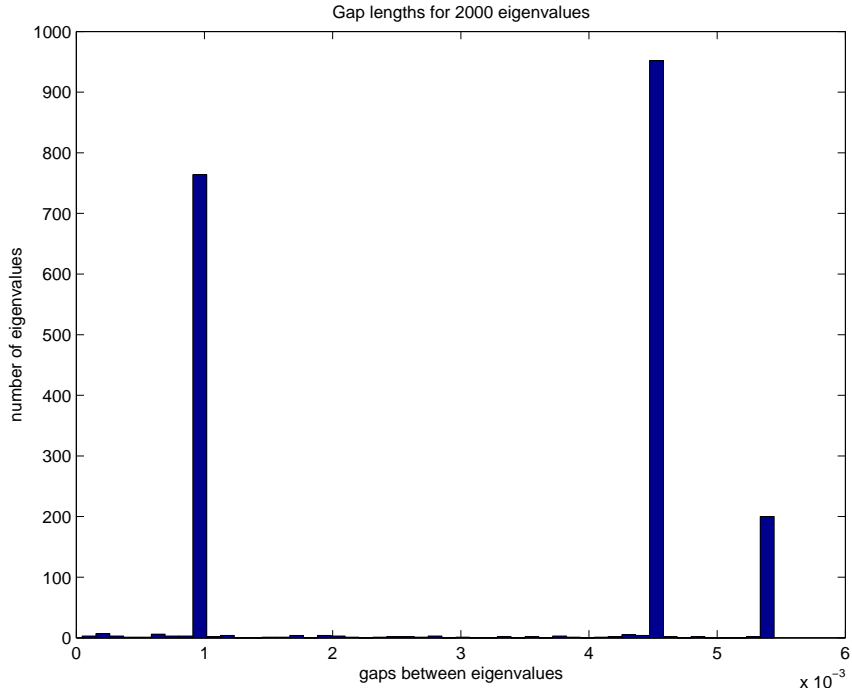}
 \caption{Zeros of $\Phi_{2000}(z; \beta)$ for $k=2$
  and $\omega = \sqrt{2}$.}
 \label{fig:1}
\end{figure}

Finally, I also want to address the zero distribution
of the para-orthogonal polynomials. This question has
not been discussed for Schr\"odinger operators yet.
Define for $\beta\in\partial\D$
\be
 \Phi_n(z; \beta) = z \Phi_{n-1}(z) - \ol{\beta} \Phi_{n-1}^{\ast}(z).
\ee
In difference to $\Phi_n(z)$ the zeros of $\Phi_n(z; \beta)$
are on the unit circle. Denote these zeros
by $\E^{2\pi\I \theta_1}, \dots, \E^{2\pi\I\theta_N}$.
An inspection of the proof of Theorem~\ref{thm:evskew}
shows that an appropriate adaption of the
results would remain true for $\Phi_n(z)$.

Before stating our main result,
I will now illustrate the behavior of the zeros with
some numerical computations. Order the values
$\theta_j$ such that
\be
 0 \leq \theta_1 < \theta_2 < \dots < \theta_N < 1.
\ee
Define the length of gaps by
\be
 g_j = \theta_{j+1} - \theta_j.
\ee
Figure~\ref{fig:1} and \ref{fig:2} show the distribution
of the values of $g_j$ for different values of $N$
when $k = 2$. One
sees that this distribution peaks at only three values.
This should remind one of the distribution of gap
lengths for the sequence of values $\{\eta n \pmod{1} \}_{n=1}^{N}$
for some value of $\eta$ and in fact, we will show this
in Theorem~\ref{thm:main4}. Also it should be pointed
out that these gap distributions do not converge.

\begin{figure}[ht]
 \includegraphics[width=0.9\textwidth]{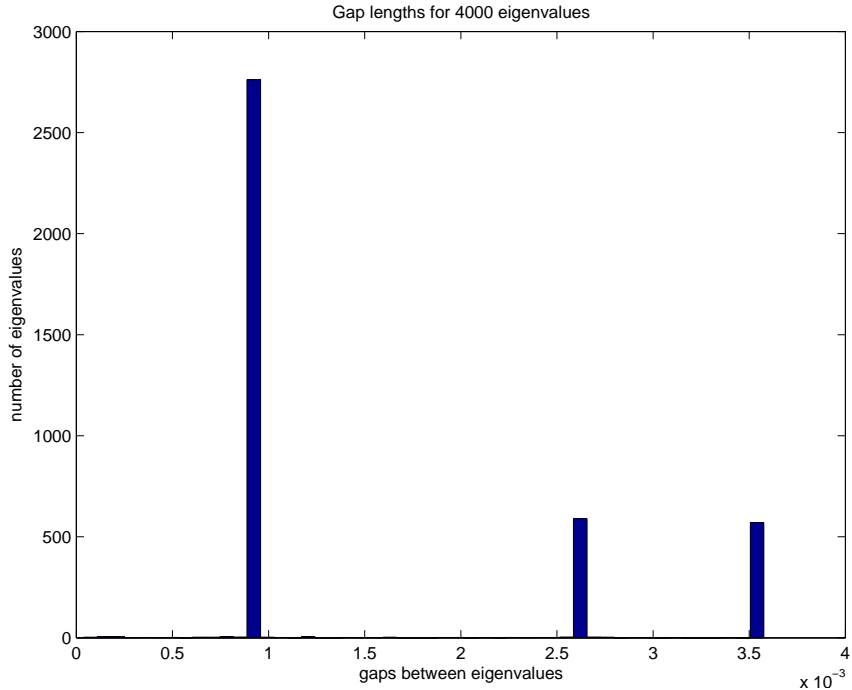}
 \caption{Zeros of $\Phi_{4000}(z; \beta)$ for $k=2$
  and $\omega = \sqrt{2}$.}
 \label{fig:2}
\end{figure}

On the other hand Figure~\ref{fig:3} shows the same
graphic for $k =3$ and the distribution resembles
an exponential distribution. One obtains similar
figures for $k\geq 4$. This is the same distribution
one would obtain if the $\theta_j$ were given by
a Poisson process and by \cite{st06} also if the
the Verblunsky coefficients $\alpha_n$ were given
by independent identically distributed random variables
whose distribution is non constant and rotationally invariant.

\begin{figure}[ht]
 \includegraphics[width=0.9\textwidth]{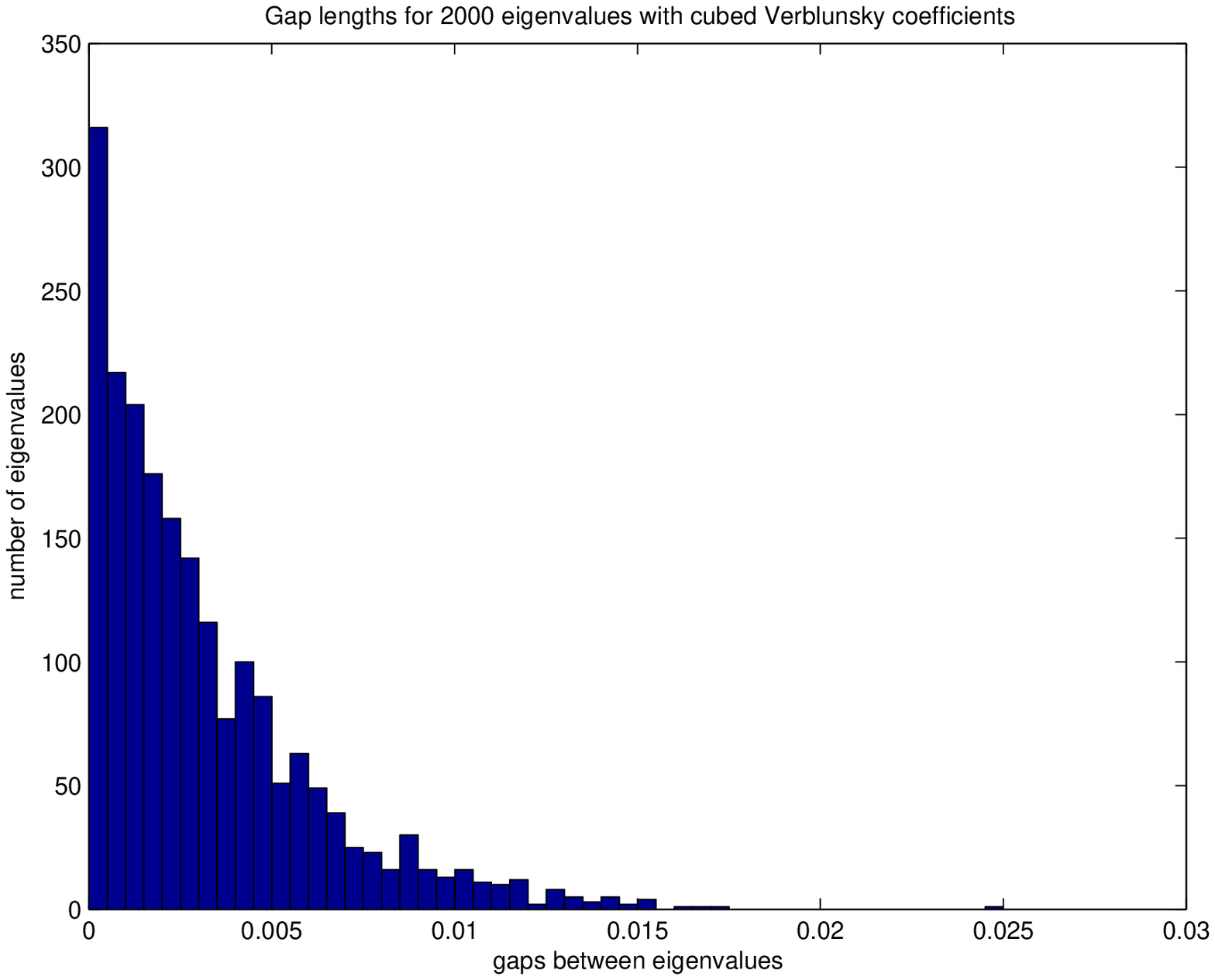}
 \caption{Zeros of $\Phi_{4000}(z; \beta)$ for $k=2$
  and $\omega = \sqrt{2}$.}
 \label{fig:3}
\end{figure}

Finally, in the case $k=1$, the (rotated) Geronimus Polynomials,
the assumptions of the Freud--Levin theorem hold
(Theorem~2.6.10 in \cite{szego}) and one has
clock spacing, so the spacing is given by the inverse
of the corresponding density of states measure.
This measure turns out to be non-constant, so
there is not a single peak.

In order to state our result, we need to
introduce more notation.
Define the Laplace functional of $N$ points
$x_1,\dots,x_N \in\T$ by
\be
 \mathfrak{L}_{\ul{x},N}(f) =
  \int_{\T} \exp\left(-\sum_{n=1}^{N} f(N x_n(\theta))
   \right)
  d\theta,
\ee
where $[-\frac{1}{2}, \frac{1}{2}) \ni 
x_n(\theta) = x_n - \theta \pmod{1}$ and $f\geq 0$
is continuous and compactly supported function.
See \cite{kisto09} for a discussion of Laplace
functionals related to zeros of paraorthogonal
polynomials.

Denote by $\mathfrak{L}^{R}_{\omega, N}$
the Laplace functional of the sequence of points
$\{n\omega \pmod{1}\}_{n=1}^{N}$. The behavior of this sequence
is well understood, see for example \cite{raven}.
In particular, this quantity does not converge to
a limit.  We will show

\begin{theorem}\label{thm:main4}
 Let $k=2$, $\tau > 1$ and assume that
 $\omega$ satisfies
 \be\label{eq:conddiop}
  \inf_{q \geq 1, q\in\Z} q^{\tau} \dist(q \omega, \Z) > 0.
 \ee
 Then for any positive, continuous, and compactly supported
 function $f:\R \to\R$, we have
 \be
  \lim_{N \to \infty} (\mathfrak{L}_{\ul{\theta}, N}(f) - 
   \mathfrak{L}^{R}_{2 \omega, N}(f)) = 0.
 \ee
\end{theorem}

This says that the values of $\mathfrak{L}_{\ul\theta, N}$
are deterministic in the large $N$ limit. However,
they do not converge to a single value as the one
for the irrational rotation does not.  
Using either Theorem~\ref{thm:main4} or easier
Theorem~\ref{thm:evskew}, one can show that the
gap distribution of the eigenvalues indeed
obeys the distribution shown in Figure~\ref{fig:1} and
\ref{fig:2}. The Diophantine
assumption \eqref{eq:conddiop} is necessary,
I sketch an argument in Remark~\ref{rem:liouville}.
Furthermore, it should be noted that Lebesgue
almost every $\omega$ satisfies \eqref{eq:conddiop}.

In this sense the case $k = 2$ is of intermediate
disorder, one has pure point spectrum with exponentially
decaying eigenfunctions, but one does not have sufficient
independence to obtain Poisson statistics.

The definition of the Laplace functional given here
is different from the one usually given in the theory
of point processes. There, one does not introduce
averaging over the unit circle by hand, but this comes
from the points $x_n$ being defined on some probability
space. In Section~\ref{sec:skew}, we will see that
our Verblunsky coefficients are defined on a probability
space, and that averaging over it in particular contains
the $\theta$ average. Hence, the name Laplace functional
is justified.

\begin{remark}\label{rem:liouville}
 Assume that for coprime integers $p,q$, $N$
 very large, and $\delta > 0$ a small parameter, we have
 that $|\omega-\frac{p}{q}|\leq \frac{1}{N^{3 + \delta}}$. 
 Then for $1 \leq n \leq N$, we have that
 \be
  \left|\alpha_n - \lambda \E^{2\pi\I \frac{p n^2}{q}}
   \right|\leq\frac{1}{N^{1 + \frac{\delta}{2}}}.
 \ee
 Since the Verblunsky coefficients $\lambda \E^{2\pi\I \frac{p n^2}{q}}$
 are $q$-periodic, the corresponding zeros of the
 paraorthogonal polynomials are clock-spaced, so
 of size $\frac{1}{N}$, whereas the points 
 $\{2 n\omega\pmod{1}\}_{n=1}^{N}$
 are all in a $\frac{1}{N^{2 + \delta}}$
 neighborhood of the points $\{\frac{\ell}{q}\}_{\ell=1}^{q}$.

 These two behaviors are clearly incompatible, and
 thus Theorem~\ref{thm:main4} cannot hold for
 Liouville frequencies.
\end{remark}

Let me now outline the rest of the content of 
the paper. Section~\ref{sec:pfmain1} discusses the
basic theory of half-line CMV matrices and gives
the proof of Theorem~\ref{thm:main1}. Then Section~\ref{sec:excmv}
introduces extended CMV operators, so ones 
defined on the whole-line, discusses restrictions
of these, defines the Green's function,
and derives useful formulas relating determinants
of CMV matrices to transfer matrices. This discussion
is somewhat more complicated than the case of
Schr\"odinger operators. Section~\ref{sec:ueCMV}
combines the formulas from the previous section
with the ones for ergodic CMV matrices.
In Section~\ref{sec:skew}, CMV matrices with built-in
rotational invariance are discussed and
Theorem~\ref{thm:main2} is proven.

In Section~\ref{sec:evstat}, we prove Theorem~\ref{thm:main4}
relying on results from Sections~\ref{sec:existtest}
and \ref{sec:greendecays}. Basically, Section~\ref{sec:greendecays}
improves the bounds on decay of the Green's function
obtained in Section~\ref{sec:ueCMV} from
unique ergodicity by using quantitative recurrence
results for the skew-shift discussed in Appendix~\ref{sec:dynskew}.
Section~\ref{sec:existtest} shows how to exploit
Section~\ref{sec:greendecays} to obtain good
test functions.

%
%
%

\section{A first look at the CMV matrix}
\label{sec:pfmain1}

In this section, we take a look at half-line CMV matrices
and provide a proof of Theorem~\ref{thm:main1}. In the following
sections, we will discuss whole line CMV matrices in more
details. Although most results in this section will be reproven
in later parts, I have included it, since it is closed
to the notation of \cite{opuc1,opuc2}.

Let $\{\alpha_n\}_{n=0}^{\infty}$ be a sequence of Verblunsky
coefficients. Define $\rho_n = (1 - |\alpha_n|^2)^{\frac{1}{2}}$
and the unitary matrices
\be
 \Theta_n = \begin{pmatrix} \ol{\alpha_n} & \rho_n \\ \rho_n
  & - \alpha_n \end{pmatrix}.
\ee
Define the operators $\mathcal{L}_+, \mathcal{M}_+$ by
\be
 \mathcal{L}_+ = \begin{pmatrix} \Theta_0 \\ & \Theta_2 \\ & & \ddots
  \end{pmatrix},\quad
 \mathcal{M}_+ = \begin{pmatrix} 1 \\ & \Theta_1 \\ & & \ddots
  \end{pmatrix}
\ee
where $1$ represents the identity $1 \times 1$ matrix.
The CMV matrix is then defined by $\mathcal{C} = \mathcal{L}_+
\mathcal{M}_+$ which will be five-diagonal and unitary. Its importance
comes from that the measure $\mu$ associated to the
Verblunsky coefficients $\{\alpha_n\}_{n=0}^{\infty}$ is
the spectral measure of $\delta_0$ with respect to
$\mathcal{C}$, so one has
\be
 \int_{\partial\D} z^n d\mu(z) =
  \spr{\delta_0}{\mathcal{C}^n \delta_0}.
\ee
We denote by $\mathrm{supp}_{\mathrm{ess}}(\mu)$ the
essential support of the measure $\mu$, that is the
support of $\mu$ with point masses removed.

\begin{lemma}\label{lem:translatealpha}
 Define $\ti\alpha_n = \alpha_{n+1}$. Let $\ti{\mu}$ be
 the measure corresponding to $\{\ti\alpha_n\}_{n=0}^{\infty}$.
 Then 
 \be
  \mathrm{supp}_{\mathrm{ess}}(\ti{\mu})
  =   \mathrm{supp}_{\mathrm{ess}}(\mu).
 \ee
\end{lemma}

\begin{proof}
 Clearly $\mathrm{supp}_{\mathrm{ess}}(\mu) = 
 \sigma_{\mathrm{ess}}(\mathcal{C})$. Let $S$ be the backward
 shift on $\ell^2(\N)$. Then $\mathcal{C}$ and
 $S^* \mathcal{\ti{C}} S$ differ by a finite rank operator.
 The claim follows.
\end{proof}

A similar proof implies that for all the translates
$\alpha^{\ell}_n = \alpha_{n+\ell}$ the corresponding
CMV matrices have the same essential spectrum.
Hence, for Verblunsky coefficients given by
\eqref{eq:defVeromegank}, one obtains that the family
of Verblunsky coefficients given by
\be
 \alpha^{\ell}_n = \lambda \exp\left(2\pi\I
  \left(\omega n^k + \sum_{j=0}^{k-1}
   \binom{k}{j} \omega \ell^{k-j} \cdot n^j\right)\right)
\ee
have the same essential spectrum. Define
for $y \in [0,1]^k$ a family of Verblunsky
coefficients by
\be
 \ti{\alpha}_{y, n} = \lambda \exp\left(2\pi\I
  \left(\omega n^k + \sum_{j=0}^{k-1}
   y_j \cdot n^j\right)\right).
\ee

\begin{lemma}\label{lem:sigmaCy}
 We have for any $y\in [0,1]^k$ that
 \be
  \sigma_{\mathrm{ess}}(\mathcal{C}) =
   \sigma_{\mathrm{ess}}(\widetilde{\mathcal{C}}_{y})
 \ee
\end{lemma}

\begin{proof}
 Given $y$, there exists a sequence $\ell_s$ such that
 \[
  \binom{k}{j} \omega \ell_s^{k-j} \to y_j
 \]
 for $0\leq j\leq k-1$ as $s\to\infty$
 (see Theorem~2.2 and Lemma~2.3 in \cite{krho1}).
 By strong convergence, one thus obtains that
 \[
  \sigma_{\mathrm{ess}}(\mathcal{C}) \supseteq
   \sigma_{\mathrm{ess}}(\widetilde{\mathcal{C}}_{y}).
 \]
 The other inclusion can be proven in a similar way.
\end{proof}

Results similar to Lemma~\ref{lem:sigmaCy} have been
discussed in \cite{lsess}.
For the proof of Theorem~\ref{thm:main1},
we will also need

\begin{proposition}\label{prop:rotatealpha}
 Define Verblunsky coefficients by
 $\ti{\alpha}_n = \E^{2\pi\I \eta n} \alpha_n$.
 Then
 \be
  \mathrm{supp}(\ti\mu) = \E^{-2\pi\I \eta}
   \mathrm{supp}(\mu).
 \ee
\end{proposition}

\begin{proof}
 This follows from the formulas in
 Appendix~A.H. in \cite{opuc2}. I will
 also give another proof in Section~\ref{sec:skew}.
\end{proof}

Given $y \in [0,1]^k$ and $\eta \in [0,1]$ define
\be
 \hat{y}_j = \begin{cases} y_j,&j=0, 2 \leq j \leq k-1;\\
   y_1 + \eta,& j=1.\end{cases}
\ee
Proposition~\ref{prop:rotatealpha} shows that
\be
 \sigma_{\mathrm{ess}}(\widetilde{\mathcal{C}}_{y}) =
  \E^{- 2\pi\I \eta}
   \sigma_{\mathrm{ess}}(\widetilde{\mathcal{C}}_{\hat{y}}).
\ee
Having this, we are now ready for

\begin{proof}[Proof of Theorem~\ref{thm:main1}]
 The results discussed so far imply that 
 $\sigma_{\mathrm{ess}}(\mathcal{C})$ is a non-empty,
 rotationally invariant, subset of $\partial\D$.
 Hence, we must have
 \[
  \sigma_{\mathrm{ess}}(\mathcal{C}) = \partial\D
 \] 
 Since also $\sigma_{\mathrm{ess}}(\mathcal{C})\subseteq
 \sigma(\mathcal{C})\subseteq\partial\D$,
 the claim follows.
\end{proof}

%
%
%

\section{Extended CMV operators}
\label{sec:excmv}

In this section, we introduce extended CMV operators
and discuss their properties that will be useful to
us. See also \cite{gzweyl} and Section~10.5 in \cite{opuc2}
for discussions from different viewpoints.

Let now $\{\alpha_n\}_{n\in\Z}$ be a bi-infinite
sequence of Verblunsky coefficients, i.e.
$\alpha_n \in \D$ although we will discuss setting
certain $\alpha_{n}$ to values in $\ol\D$ below.
Recall that $\rho_n = (1-|\alpha_n|^2)^{\frac{1}{2}}$ and
\be
 \Theta_n = \begin{pmatrix} \ol{\alpha_n} & \rho_n \\
  \rho_n & - \alpha_n \end{pmatrix}
\ee
viewed as acting on $\ell^2(\{n,n+1\})$. Define
\be
 \mathcal{L} = \bigoplus_{n\text{ even}} \Theta_n,\quad
  \mathcal{M} = \bigoplus_{n\text{ odd}} \Theta_n
\ee
and the extended CMV operator $\mathcal{E} = \mathcal{L}
 \cdot \mathcal{M}$. We note

\begin{lemma}
 $\mathcal{E}$, $\mathcal{L}$, and $\mathcal{M}$ are
 unitary operators $\ell^2(\Z)\to\ell^2(\Z)$. Furthermore,
 $\mathcal{L}$ leaves the subspaces $\ell^2(\{n,n+1\})$
 for $n$ even invariant, whereas $\mathcal{L}$ does this for
 $n$ odd.
\end{lemma}

We will now discuss various restrictions of CMV
operators. First denote by $P^{[a,b]}$ the projection
$\ell^2(\Z) \to \ell^2([a,b])$. We define
\be
 X^{[a,b]} = (P^{[a,b]})^* X P^{[a,b]}
\ee
for $X\in\{\mathcal{E},\mathcal{M},\mathcal{L}\}$.

\begin{lemma}
 $\mathcal{E}^{[a,b]} = \mathcal{L}^{[a,b]} \mathcal{M}^{[a,b]}$.
\end{lemma}

\begin{proof}
 Compute.
\end{proof}

It is easy to check that the operator $\mathcal{E}^{[a,b]}$
will no longer be unitary, but it will still be an useful
object. Let now $\beta \in \partial\D$ and $a\in\Z$
and consider the modified Verblunsky coefficients
\be
 \ti{\alpha}_n = \begin{cases} \alpha_n,& n\neq a\\
  \beta,& n=a.\end{cases}
\ee
We then have that $\widetilde{\mathcal{E}}$,
$\widetilde{\mathcal{L}}$, and
$\widetilde{\mathcal{M}}$ leave the
spaces $\ell^2(\{a+1, a+2, \dots\})$ and $\ell^2(\{\dots, a-1,a\})$
invariant. In particular, we can define unitary
restrictions
\be
 \mathcal{E}^{[a+1, \infty)}_{\beta, \bullet} = P^{[a+1,\infty)}
  \widetilde{\mathcal E} P^{[a+1, \infty)},\quad
 \mathcal{E}^{(-\infty,a]}_{\bullet, \beta} = P^{(-\infty,a]}
  \widetilde{\mathcal E} P^{(-\infty,a]}.
\ee

\begin{lemma}\label{lem:relationC}
 Let $\mathcal{C}$ be the CMV operator with Verblunksy
 coefficients $\{\alpha_n\}_{n=0}^{\infty}$. Then
 \be
  \mathcal{C} = \mathcal{E}^{[0,\infty)}_{1, \bullet}.
 \ee
 Denote by $R$ the identification $\ell^2(\{\dots, -2,-1\})$
 with $\ell^2(\{0,1,2,\dots\})$ and
 by $\mathcal{C}_-$ the CMV operator with Verblunsky
 coefficients $\{-\ol{\alpha}_{-n-1}\}_{n=0}^{\infty}$.
 Then
 \be
  \mathcal{C}_- =  R \mathcal{E}^{(-\infty,-1]}_{\bullet, 1} R^{*}.
 \ee
\end{lemma}

\begin{proof}
 These are computations.
\end{proof}

We will now consider restrictions to intervals.
So let $a < b$ be integers, and $\beta,\gamma \in
\partial\D$. Define a sequence of Verblunsky
coefficients
\be
 \ti{\alpha}_n = \begin{cases} \beta,&n=a;\\
  \gamma,&n=b;\\
   \alpha_n,& n\notin\{a,b\} \end{cases}.
\ee
We then define the operator
\be
 \mathcal{E}^{[a+1,b]}_{\beta, \gamma}
  = P^{[a+1,b]} \widetilde{\mathcal{E}} P^{[a+1,b]}.
\ee
Of course, this definition makes sense
for $\beta,\gamma \in \ol{\D}$ and $a = - \infty$
or $b = \infty$. Furthermore, we write
$\bullet$ if we leave $\alpha_a$ or $\alpha_b$
unchanged to match the previous definition.
$\beta,\gamma\in\partial\D$ should be thought of
as boundary conditions.

\begin{lemma}
 If $\beta,\gamma\in\partial\D$ then 
 $\mathcal{E}^{[a,b]}_{\beta,\gamma}$,
 $\mathcal{L}^{[a,b]}_{\beta,\gamma}$, and
 $\mathcal{M}^{[a,b]}_{\beta,\gamma}$ are unitary.
\end{lemma}

Since the equation $\mathcal{E} \psi = z \psi$
is equivalent to $(z \mathcal{L}^{\ast} - \mathcal{M}) \psi = 0$.
We note for further reference

\begin{lemma}\label{lem:tridiag}
 The matrix $A = z (\mathcal{L}^{[a,b]}_{\beta,\gamma})^{*}
  - \mathcal{M}^{[a,b]}_{\beta,\gamma}$ is tridiagonal.
 Write $A = \{A_{i,j}\}_{a\leq i,j \leq b}$.
 Then we have that
 \be
  A_{j,j} = \begin{cases}
   z \alpha_j + \alpha_{j-1},&j\text{ even}\\
  - z \ol{\alpha_{j-1}}
   - \ol{\alpha_{j}},&j\text{ odd},\end{cases}\quad
  A_{j+1, j} = A_{j,j+1} = \ti{\rho}_j = 
  \begin{cases} z \rho_j,&j\text{ even},\\
   - \rho_{j},&j\text{ odd}.\end{cases}
 \ee
\end{lemma}

Let $z\in\C$, $\beta,\gamma\in\partial{D}$,
$a\leq k,\ell\leq b$, then the Green's function is 
defined by
\be
 G^{[a,b]}_{\beta,\gamma}(z; k,\ell)=
  \spr{\delta_k}{(z \left(\mathcal{L}^{[a,b]}_{\beta,\gamma})^{*}
   - \mathcal{M}^{[a,b]}_{\beta,\gamma}\right)^{-1} \delta_\ell}.
\ee
Our goal now will be to provide a formula
for the Green's function in terms of quantities
that are easier to analyze, like the formula
for the Green's function of Schr\"odinger operators
in term of orthogonal polynomials, respectively
entries of the transfer matrix.

We define
\begin{align}
 \Phi^{[a,b]}_{\beta,\gamma}(z) 
 &= \det\left(z - \mathcal{E}^{[a,b]}_{\beta,\gamma}\right) \\
\nn &= \det\left(
  z (\mathcal{L}^{[a,b]}_{\beta,\gamma} )^*
 - \mathcal{M}^{[a,b]}_{\beta,\gamma}
   \right) \cdot \det((\mathcal{L}^{[a,b]}_{\beta,\gamma} )^*)
\end{align}
and
\be
 \varphi^{[a,b]}_{\beta,\gamma}(z) = (\rho_a \cdots \rho_{b})^{-1}
  \Phi^{[a,b]}_{\beta,\gamma}(z).
\ee

\begin{lemma}
 Let $\Phi_n(z)$ be defined as in \eqref{eq:defPhin}.
 Then
 \be
  \Phi_{n}(z) = \Phi^{[0,n-1]}_{1,\bullet}(z).
 \ee
\end{lemma}

\begin{proof}
 Proposition~3.4. in \cite{scmv5} states 
 \[
  \Phi_n(z) = \det(z - \mathcal{E}^{[0,n-1]}_{1, \bullet}).
 \]
 The claim follows.
\end{proof}

We also introduce the Aleksandrov polynomials $\Phi^{\beta}_n(z)$
by applying the recursion \eqref{eq:defPhin} to
the Verblunsky coefficients $\{\beta \alpha_n\}_{n=0}^{\infty}$.
In particular, the polynomial of the second
kind is defined by
\be
 \Psi_n(z) = \Phi^{-1}_n(z).
\ee
We have that (Theorem~9.5. in \cite{s1foot})

\begin{lemma}
 We have
 \be
  \Phi^{\beta}_n(z) = \Phi_{\beta,\bullet}^{[0,n-1]}(z)
 \ee
 and
 \be
  \Phi^{\beta}_n(z;\gamma) = 
   \Phi_{\beta,\gamma}^{[0,n-1]}(z)
 \ee
\end{lemma}

With these formulas, we obtain the following equality
for the absolute value of the Green's function. It would
be possible to derive an equality for the Green's function
but one would need distinguish between 4 cases depending
on if $a$ or $b$ is even or odd.

\begin{proposition}\label{prop:greenaspolynomial}
 Let $z\in\C$, $\beta,\gamma\in\partial\D$,
 and $a \leq k \leq \ell \leq b$. Then
 \be
  \left| G^{[a, b]}_{\beta,\gamma} (z; k, \ell)\right|
   = \frac{1}{\rho_k \rho_{\ell}}
   \left| \frac{\varphi_{\beta, \bullet}^{[a,k-1]}(z)
    \varphi_{\bullet, \gamma}^{[\ell+1, b]}(z)}
  {\varphi_{\beta,\gamma}^{[a,b]}(z)}\right|
 \ee
\end{proposition}

\begin{proof}
 By Cramer's rule and Lemma~\ref{lem:tridiag},
 we thus obtain
 \[
  \left|G^{[a, b]}_{\beta,\gamma} (z; k, \ell)\right|
   = \ti{\rho}_{k+1} \cdots \ti{\rho}_{\ell - 1}
    \left|\frac{\Phi_{\beta, \bullet}^{[a,k-1]}(z)
    \Phi_{\bullet, \gamma}^{[\ell+1, b]}(z)}
  {\Phi_{\beta,\gamma}^{[a,b]}(z)}\right|
 \]
 The claim now follows from the definition of $\varphi$.
\end{proof}

This formula is more awkward than the one for
Schr\"odinger operators, since it involves
three different type of polynomials whereas the
one for Schr\"odinger operators only has one
(see (2.7) in \cite{bbook}). Nevertheless it
is useful in exactly the same way.
We now give the relation of the Green's function
to solution of our equation.

\begin{lemma}\label{lem:green2sol}
 Let $\psi$ solve $\mathcal{E} \psi = z \psi$. Then
 for $a < n < b$
 \begin{align}
  \psi(n) & = G^{[a,b]}_{\beta,\gamma}(z;n,a)
   \begin{cases} (z \ol{\beta} - \alpha_a) \psi(a)
    -\rho_a \psi(a+1),& a\text{ even};\\
   (z \alpha_a - \beta) \psi(a) + z \rho_a \psi(a+1),&a\text{ odd}
    \end{cases}\\
\nn  &+ G^{[a,b]}_{\beta,\gamma}(z;n, b)
   \begin{cases} (z \ol{\gamma} - \alpha_b) \psi(b)
    -\rho_b \psi(b-1),& b\text{ even};\\
   (z \alpha_b - \gamma) \psi(b) 
    + z \rho_{b-1} \psi(b-1),&b\text{ odd}
    \end{cases}
 \end{align}
\end{lemma}

\begin{proof}
 With $A = (z(\mathcal{L}^{[a,b]}_{\beta,\gamma})^{\ast} - 
    \mathcal{M}^{[a,b]}_{\beta,\gamma})$, we have
 \[
  \varphi(n) = \spr{A^{-1} \delta_n, A \varphi}.
 \]
 Since, $(z(\mathcal{L})^{\ast} - \mathcal{M} )\varphi=0$,
 we have that for $a+1 \leq n \leq b$ also
 \[
  A \varphi (n) = 0.
 \]
 The claim now follows by evaluating this
 expression for $n \in \{a,b\}$.
\end{proof}

Our next goal will be to introduce transfer
matrices and related them to the determinants
defined above. We begin with the one-step
transfer matrix
\be
 A_z(\alpha) = \frac{1}{(1-|\alpha|^2)^{\frac{1}{2}}}
  \begin{pmatrix} z & - \ol{\alpha} \\
  -\alpha z & 1 \end{pmatrix}.
\ee
We define the transfer matrix by
\be
 T^{[a,b]}(z) = A_z(\alpha_{b}) \cdots A_z(\alpha_a).
\ee

\begin{lemma}
 We have that
 \be
  T^{[a,b]}(z) = \frac{1}{2} \begin{pmatrix}
   \varphi^{[a,b]}_{1,\bullet}(z) + 
    \varphi^{[a,b]}_{-1,\bullet}(z) & 
   \varphi^{[a,b]}_{1,\bullet}(z) -  
    \varphi^{[a,b]}_{-1,\bullet}(z) \\
   (\varphi^{[a,b]}_{1,\bullet})^*(z) - 
    (\varphi^{[a,b]}_{-1,\bullet})^*(z) &  
   (\varphi^{[a,b]}_{1,\bullet})^*(z)
    + (\varphi^{[a,b]}_{-1,\bullet})^*(z)\end{pmatrix}.
 \ee
 where $(\varphi^{[a,b]}_{\beta,\gamma})^*(z)
  = z^{b-a+1} \ol{\varphi^{[a,b]}_{\beta,\gamma}(\ol{z}^{-1})}$.
\end{lemma}

\begin{proof}
 The $T_n(z)$ in \cite{s1foot} is $T^{[0,n-1]}(z)$
 in our notation. We have that
 \[
  T_n(z) = \frac{1}{2} \begin{pmatrix}
   \varphi_n(z) + \psi_n(z) & \varphi_n(z) - \psi_n(z) \\
   \varphi_n^*(z) - \psi_n^*(z) &  \varphi_n^*(z)
    + \psi_n^*(z)\end{pmatrix}.
 \]
 It follows that
 \[
  T^{[0,n-1]}(z) = \frac{1}{2} \begin{pmatrix}
   \varphi^{[0,n-1]}_{1,\bullet}(z) + 
    \varphi^{[0,n-1]}_{-1,\bullet}(z) & 
   \varphi^{[0,n-1]}_{1,\bullet}(z) -  
    \varphi^{[0,n-1]}_{-1,\bullet}(z) \\
   (\varphi^{[0,n-1]}_{1,\bullet})^*(z) - 
    (\varphi^{[0,n-1]}_{-1,\bullet})^*(z) &  
   (\varphi^{[0,n-1]}_{1,\bullet})^*(z)+ 
   (\varphi^{[0,n-1]}_{-1,\bullet})^*(z)\end{pmatrix}.
 \]
 The claim follows using translation invariance.
\end{proof}

We thus obtain that

\begin{corollary}
 We have that
 \be\label{eq:varphibetabullet}
  \begin{pmatrix} \varphi^{[a,b]}_{\beta,\bullet}(z)\\
  (\varphi^{[a,b]}_{\beta,\bullet})^{*}(z) \end{pmatrix} = 
   T^{[a,b]}(z) \begin{pmatrix} 1 \\ \ol{\beta} \end{pmatrix}
 \ee
 and
 \be\label{eq:varphibetagamma}
  \varphi^{[a,b]}_{\beta,\gamma}(z) = 
   \frac{1}{\rho_b}
    \spr{\begin{pmatrix} z \\ - \ol{\gamma} \end{pmatrix}}
    {T^{[a,b-1]}(z) \begin{pmatrix} 1 \\ \ol{\beta} \end{pmatrix}}.
 \ee
\end{corollary}

\begin{proof}
 The first equation is (3.2.26) in \cite{opuc1}.
 For the second equation, we have that
 \[
  \Phi^{[a,b]}_{\beta,\gamma}(z) = z
   \Phi^{[a,b-1]}_{\beta,\bullet}(z) - \ol{\gamma}
    (\Phi^{[a,b-1]}_{\beta,\bullet})^*(z).
 \] 
 We thus have that
 \[
  \varphi^{[a,b]}_{\beta,\gamma}(z) = 
   \frac{1}{\rho_b} \left( z 
    \varphi^{[a,b-1]}_{\beta,\bullet}(z) - \ol{\gamma}
    (\varphi^{[a,b-1]}_{\beta,\bullet})^*(z)
      \right),
 \]
 which implies the second equation by the first one.
\end{proof}

There is one final object, we need to identify
$\varphi^{[a,b]}_{\bullet, \gamma}(z)$. We employ
the same strategy as we used in Lemma~\ref{lem:relationC}
to identify $\mathcal{E}^{(-\infty,0]}_{\bullet,\gamma}$.
Let 
\be
 \ti{\alpha}_{n} = \begin{cases} \gamma, & n = -1;\\
   - \ol{\alpha}_{b-n}, &n \geq 0.\end{cases}
\ee
Then we have that
\be
 \varphi^{[a,b]}_{\bullet, \gamma}(z) =
  \ti{\varphi}^{[0,b-a-1]}_{\gamma,\bullet}(z).
\ee

\begin{lemma}
 We have htat
 \be\label{eq:varphibulletgamma}
  \begin{pmatrix} \varphi^{[a,b]}_{\bullet,\gamma}(z) \\
   (\varphi^{[a,b]}_{\bullet,\gamma})^*(z) \end{pmatrix} =
   \begin{pmatrix} - \frac{1}{z} & 0 \\ 0 & 1 \end{pmatrix}
    (T^{[a,b]}(z))^t \begin{pmatrix} -z\\\ol{\gamma} \end{pmatrix}.
 \ee
\end{lemma}

\begin{proof}
 We have that
 \[ 
  \begin{pmatrix} - \frac{1}{z} & 0 \\ 0 & 1 \end{pmatrix} 
  A_z(-\ol{\alpha})^{t} \begin{pmatrix} -z & 0 \\ 0 & 1 \end{pmatrix} = A_z(\alpha).
 \] 
 From this the claim follows.
\end{proof}

%
%
%

\section{Strictly ergodic CMV matrices}
\label{sec:ueCMV}

In this section, we will consider families of CMV operators.
This has the advantage that certain formulas will simplify,
when viewed probabilistically. Also strict ergodicity
simplifies certain statements not available in the ergodic
case, in particular \cite{furman}.

Let $\Omega$ be a compact metric space,
$T:\Omega\to\Omega$ a uniquely ergodic
and minimal homeomorphism,
and $\mu$ the unique $T$-invariant probability measure.
We call $(\Omega,\mu,T)$ strictly ergodic in this case.
For a continuous function $f:\Omega\to\D$, we define
the family of Verblunsky coefficients
\be
 \alpha_{\omega, n} = f(T^n \omega).
\ee
We denote by $\mathcal{E}_{\omega}, \dots$ the associated
objects.

The main example to keep in mind is the $k$-dimensional
skew-shift with $\Omega = \T^k = (\R/\Z)^k$
\be\label{eq:defkskew}
 (T x)_{\ell} = \begin{cases} x_1 + \omega, & \ell = 1; \\
   x_{\ell} + x_{\ell - 1}, &2 \leq \ell \leq k.\end{cases}
\ee
One can then show by induction that
\be
 (T^n x)_{\ell} = \binom{n}{\ell} \omega + \binom{n}{\ell -1} x_1
   + \dots + \binom{n}{0} x_{\ell}.
\ee
This map is strictly ergodic, see Proposition~4.7.4. in \cite{brinstuck}.
Then one can realize the Verblunsky coefficient
from the introduction as $\alpha_{x, n}$
for $f(x) = \lambda \E^{2\pi\I x_k}$ and a particular
choice of $x$.

We now return to our study of the general case
of uniquely ergodic and minimal CMV matrices.

\begin{lemma}
 We have that $\mathcal{E}_{T x} = (S^{\ast} \mathcal{E}_x S)^{t}$,
 where $S$ is the usual forward shift on $\ell^2(\Z)$. In
 particular for any $x,y\in\Omega$
 \be
  \sigma(\mathcal{E}_x) = \sigma(\mathcal{E}_y).
 \ee
\end{lemma}

\begin{proof}
 The first claim is algebraic. The second
 claim follows as Lemma~\ref{lem:sigmaCy}.
\end{proof}

For $n\geq 1$, we define the $n$-step (forward) transfer matrix by
\be\label{eq:defTxn}
 T_{x;n}(z) = A(\alpha_{x,n-1}, z) \cdots A(\alpha_{x,0}, z).
\ee
We note that $T_{x;n}(z) = T^{[0,n-1]}_{x}(z)$ in the notation
of the previous section, and that also
$T^{[a,b]}_x(z) = T_{T^{a} x; b-a+1}(z)$.
The Lyapunov exponent is defined by
\be
 L(z) = \lim_{n\to\infty} \frac{1}{n} \int_{\T^k}
  \log\|T_{n,x}(z)\| dx.
\ee
We collect its properties

\begin{proposition}\label{prop:proplyap}
 Let $(\Omega,\mu,T)$ be strictly ergodic and
 $z\in\partial\D$.
 \begin{enumerate}
  \item $L(z) \geq 0$.
  \item For almost-every $x\in\T^k$, we have as $n\to\infty$
   that
   \be
    \frac{1}{n} \log\|T_{x; n}(z)\| \to L(z).
   \ee
  \item For every $\eps > 0$, there exists $N$ such that
   for $n \geq N$ and $x\in\T^K$ we have
   \be
    \frac{1}{n} \log\|T_{x; n}(z)\| \leq L(z) + \eps.
   \ee
 \end{enumerate}
\end{proposition}

\begin{proof}
 (i) follows from $\det(A(\alpha,z)) = z$.
 (ii) is the subadditive ergodic theorem (see Corollary 10.5.25
  in \cite{opuc2}).
 (iii) is Furman's strengthening for uniquely
 ergodic transformations \cite{furman}.
\end{proof}

The right extension of \eqref{eq:defTxn}
for negative numbers is
\be
 T_{x; - n}(z) = \begin{pmatrix} -\frac{1}{z} & 0 \\ 0 & 1
   \end{pmatrix} A(-\ol{\alpha_{x, -1}}, z) \cdots
  A(-\ol{\alpha_{x,-n}}, z) \begin{pmatrix} 
   -z & 0 \\ 0 & 1 \end{pmatrix}
\ee
(where $n\geq 0$). This can be seen from 
\eqref{eq:varphibulletgamma}. In particular,
one has
\be
 L(z) = \lim_{n\to\infty} \frac{1}{n} \int_{\T^k}
  \log\|T_{x; -n}(z)\|dx.
\ee

\begin{lemma}
 Let $(\Omega,\mu,T)$ be strictly ergodic
 and $\eps > 0$.
 There exists $C > 1$ such that for $n \geq 1$
 and $\beta,\gamma \in \partial\D$,
 we have for $0\leq k \leq \ell \leq n-1$ that
 \be
  |G_{x; \beta,\gamma}^{[0,n-1]}(z; k,\ell)| \leq
   C \frac{\E^{ (L(z) + \eps) (k + n - 1 - \ell)}}
   {|\varphi^{[0,n-1]}_{x; \beta,\gamma}(z)|}. 
 \ee 
\end{lemma}

\begin{proof}
 By (iii) of Proposition~\ref{prop:proplyap}, there exists $c \geq 1$
 such that for any $x\in\Omega$ and $n\geq 1$, we have
 \[
  \|T_{x; n}(z)\| \leq c \E^{(L(z) + \eps) n}.
 \]
 By \eqref{eq:varphibetabullet} and \eqref{eq:varphibulletgamma},
 we obtain that the numerator in Proposition~\ref{prop:greenaspolynomial}
 is bounded by
 \[
  c_2 \cdot \E^{(L(z) + \eps) (n-1 - \ell + k)}.
 \]
 The claim follows.
\end{proof}

In particular, we obtain the important theorem

\begin{theorem}\label{thm:uegreen}
 Let $(\Omega,\mu,T)$ be strictly ergodic,
 $m \in (0, L(E))$, $\delta > 0$,
 and $\beta_0, \gamma_0\in\partial\D$.
 Then for $n$ large enough, there exists $\Omega_n$ 
 satisfying $\mu(\Omega_n) \geq 1 - \delta$ and
 for $x \in \Omega_n$ there exists
 \be
  \beta \in \{-\beta_0,\beta_0\},\quad 
   \gamma\in\{-\gamma_0,\gamma_0\}
 \ee
 such that
 for $\frac{n}{3} \leq \ell \leq \frac{2n}{3}$
 and $k \in \{0,n-1\}$
 \be
  |G_{x; \beta,\gamma}^{[0,n-1]}(z; k,\ell)| \leq
   \E^{- m |k - \ell|}.
 \ee
\end{theorem}

\begin{proof}
 By \eqref{eq:varphibetagamma}, we have that
 \[
  \begin{pmatrix} \varphi^{[0,n-1]}_{x;\beta,\gamma}(z) &
   \varphi^{[0,n-1]}_{x;-\beta,\gamma}(z) \\
   \varphi^{[0,n-1]}_{x;\beta,-\gamma}(z) &
   \varphi^{[0,n-1]}_{x;-\beta,-\gamma}(z) \end{pmatrix}
   = \begin{pmatrix} z & -\ol{\gamma} \\ z & \ol{\gamma}
     \end{pmatrix}
   T_{x,n}(z) \begin{pmatrix} 1 & 1 \\ \beta & -\beta \end{pmatrix}
 \]
 Since for almost every $x$
 $\frac{1}{n}\log\|T_{x,n}(z)\| \geq L(z) (1 - \eps)$
 for $n$ large enough, the claim
 follows.
\end{proof}

%
%
%

\section{Rotationally invariance and the proof of Theorem~\ref{thm:main2}}
\label{sec:skew}

We begin this section by investigating what
happens if one rotates the Verblunsky coefficients,
which is essentially what we used to prove Theorem~\ref{thm:main1}.
We have the following important proposition

\begin{proposition}\label{prop:rotatealphas}
 Let $\beta,\gamma\in\ol\D$, $a<b$ integers, and
 $x,y\in\T$ and define
 \be
  \ti{\alpha}_n = e(n x+y) \alpha_n,\quad
   \ti{\beta} = e((a-1) x+ y) \beta,\quad
    \ti{\gamma} = e(b x + y) \gamma.
 \ee
 Then $\mathcal{E}^{[a,b]}_{\beta,\gamma}$
 and $e(x) \widetilde{\mathcal{E}}^{[a,b]}_{\ti{\beta},\ti{\gamma}}$
 are unitarily equivalent.
\end{proposition}

Here and in the following, we abbreviate
$e(x) = \E^{2\pi\I x}$.
We will prove this proposition in the case of $a$ and
$b$ finite. It is interesting if it holds for $a,b$
possibly infinite. An inspection of the proof of
Proposition~\ref{prop:rotatealphas} shows that it
also holds for whole line CMV operators with pure
point spectrum. In particular, it implies that
in the case $k = 2$, all the operators $\mathcal{E}_x$
defined by the skew-shift are unitarily equivalent.
Since the Jitomirskaya--Simon \cite{js} argument
applies in our case, all the $\mathcal{E}_x$ have purely
singular continuous spectrum.
For the proof of this proposition, we need the following 
lemma

\begin{lemma}\label{lem:rotatealpha}
Pick some $u_a\in\partial\D$ and define a sequence
 recursively by
 \be
  u_{n} = \begin{cases} u_{n-1} e(-(n-1) x - y),&n\text{ even};\\
    u_{n-1} e((n-1) x+y),&n\text{ odd}.\end{cases}
 \ee
 Furthermore, we define the multiplication operators
 \be
  U \psi(n) = u_n \psi(n),\quad V \psi(n) = \begin{cases}
   u_{n-1} \psi(n), & n\text{ even};\\
   u_{n-1} e(-x) \psi(n), & n\text{ odd}\end{cases}.
 \ee
 Then for $\ti{z} = e(-x) z$
 \be
  \left(\ti{z} (\widetilde{\mathcal{L}}^{[a,b]}_{\ti\beta,\ti\gamma})^*
   - \widetilde{\mathcal{M}}^{[a,b]}_{\ti\beta,\ti\gamma}\right) U =
  V \left(z (\mathcal{L}^{[a,b]}_{\beta,\gamma})^*
   - \mathcal{M}^{[a,b]}_{\beta,\gamma}\right).
 \ee
\end{lemma}

\begin{proof}
 A computation shows for $n$ even that
 \[
  \ti{z} \ti{\alpha}_n + \ti{\alpha}_{n-1} = e((n-1)x+y) (z\alpha_n+\alpha_{n-1})
 \]
 and for $n$ odd
 \[
  \ti{z} \ol{\ti{\alpha}_{n-1}} + \ol{\ti{\alpha}_{n}} = 
  e(-nx-y) (z\ol{\alpha_{n-1}}+\ol{\alpha_{n}}).
 \]
 By Lemma~\ref{lem:tridiag}, we thus obtain that for $n$ even
 we have that
 \begin{align*}
  (\ti{z} (\widetilde{\mathcal{L}}^{[a,b]}_{\beta,\gamma})^*
   - \widetilde{\mathcal{M}}^{[a,b]}_{\beta,\gamma}) U \psi (n) &=
   e(-x) z \rho_{n} u_{n+1} \psi(n+1) - \rho_{n-1} u_{n-1} \psi(n-1)\\
& + u_n e((n-1) x +y) (z\alpha_n+\alpha_{n-1}) \psi(n).
 \end{align*}
 Since $u_{n} = e(-(n-1) x - y) u_{n-1}$ and $u_{n+1} = e(x)u_{n-1}$,
 the claimed equality follows for $n$ even.
 Similarly for $n$ odd
 \begin{align*}
  (\ti{z} (\widetilde{\mathcal{L}}^{[a,b]}_{\beta,\gamma})^*
   - \widetilde{\mathcal{M}}^{[a,b]}_{\beta,\gamma}) U \psi (n) &=
   - \rho_n u_{n+1}  \psi_{n+1} + e(-x) z \rho_{n-1} u_{n-1} \psi_{n-1} \\
  & - e(-nx - y) (z \ol{\alpha_{n-1}} + \ol{\alpha_{n}}) u_n \psi(n).
 \end{align*} 
 Since $u_{n+1} = u_n \cdot e(-nx-y)$ and $u_{n} = e((n-1)x + y) u_{n-1}$,
 we obtain the claim.
\end{proof}

\begin{proof}[Proof of Proposition~\ref{prop:rotatealphas}]
 Since the spectra of $\mathcal{E}^{[a,b]}_{\beta,\gamma}$
 and $e(x) \widetilde{\mathcal{E}}^{[a,b]}_{\ti\beta,\ti\gamma}$ are
 simple, it suffices to show that they are the same.
 If $\mathcal{E}^{[a,b]}_{\beta,\gamma} \psi = z \psi$
 for $\psi\neq 0$,
 we have that $(z (\mathcal{L}^{[a,b]}_{\beta,\gamma})^*
   - \mathcal{M}^{[a,b]}_{\beta,\gamma}) \psi = 0$.
 Hence, by the previous lemma also that
 \[
  (\ti{z} (\widetilde{\mathcal{L}}^{[a,b]}_{\ti\beta,\ti\gamma})^*
   - \widetilde{\mathcal{M}}^{[a,b]}_{\ti\beta,\ti\gamma}) \varphi =0
 \]
 for $\varphi = U \psi \neq 0$. Hence, we also have that
 \[
  (z - e(x) \widetilde{\mathcal{E}}^{[a,b]}_{\ti\beta,\ti\gamma}) \varphi =0 
 \]
 which implies the claim.
\end{proof}

We will now begin drawing conclusions from Proposition~\ref{prop:rotatealphas}.
For the sake of concreteness, we will only consider
the Verblunsky coefficients given by
\be
 \alpha_{x,n} = \lambda \E^{2\pi\I (T^n x)_k}
\ee
where $x\in\T^k$, $\lambda\in\D\setminus\{0\}$,
and $T:\T^k\to\T^k$ is the
$k$ dimensional skew-shift defined in \eqref{eq:defkskew}.

For $\theta_1, \theta_2 \in \T$, we denote by
$P_{[\theta_1,\theta_2]}$ the spectral projection
on the arc $\{\E^{2\pi\I t}:\quad t \in [\theta_1, \theta_2] \pmod{1}\}$.
We then have that

\begin{theorem}\label{thm:wegneresti}
 Let $\beta_0,\gamma_0 \in\partial\D$ and define
 \be
  \beta_x = \beta_0 \frac{\alpha_{x,-1}}{|\alpha_{x,-1}|},\quad 
   \gamma_x = \gamma_0 \frac{\alpha_{x,n-1}}{|\alpha_{x,n-1}|}.
 \ee
 Then
 \be
  \frac{1}{n} \int_{\T^k}\tr\left(P_{[\vartheta_1, \vartheta_2]}
   \mathcal{E}^{[0,n-1]}_{x; \beta_x,\gamma_x} \right) dx
  = |\theta_2 - \theta_1|.
 \ee
\end{theorem}

\begin{proof}
 We will show this is true, when only performing the $x_{k-1}$
 integral. Let $s = x_{k-1}$. Then changing $s$ amounts to changing
 $x$ in Proposition~\ref{prop:rotatealphas}. Hence, the eigenvalues
 are given by
 \[
  \E^{2\pi\I (\theta_1 - s)},\dots, \E^{2\pi\I (\theta_N - s)}
 \]
 as $s$ varies. This implies the claim.
\end{proof}

It is easy to infer from this that the integrated
density of states is just given by the normalized
Lebesgue measure.
We now come to 

\begin{theorem}\label{thm:lyappos}
 For $z\in \partial\D$, we have that
 \be
  \gamma(z) = - \frac{1}{2} \log(1 - |\lambda|^2).
 \ee
\end{theorem}

\begin{proof}
  This can be shown as in Theorem~12.6.2. in \cite{opuc2}.
\end{proof}

\begin{proof}[Proof of Theorem~\ref{thm:main2}]
 For $\theta\in\T$, we have
 \[
  \alpha_{\ti{x}, n} = \E^{2\pi\I \theta} \alpha_{x, n}
 \]
 where
 \[
  \ti{x}_{\ell} = \begin{cases}
   x_{\ell}, &1\leq \ell \leq k-1; \\
   x_{k} + \theta,&\ell=k.\end{cases}.
 \]
 The claim now follows from Theorem~12.6.1.
 in \cite{opuc2}.
\end{proof}

\begin{proof}[Proof of Proposition~\ref{prop:rotatealpha}]
 If $z \in \sigma_{\mathrm{ess}}(\mathcal{C})$ then there
 exists a sequence $\psi_j \in \ell^2(\N)$ such that
 $\|\psi_j\| = 1$,  $\psi_j \to 0$ weakly,
 and $\|(\mathcal{C} - z)\psi_j\| \to 0$. In particular,
 we have for any $N \geq 1$ fixed
 \[
  \sum_{n=1}^{N} |\psi_j(n)|^2 \to 0.
 \]
 By Lemma~\ref{lem:rotatealpha} with $x = 2\pi\eta$,
 $y = 0$, we obtain that $\varphi_j = U \psi_j$
 satisfy $\varphi_j \to 0$ weakly and
 \[
  \|(\widetilde{\mathcal{C}} - \E^{-2\pi\I \eta} z) \varphi_j\|
   \to 0.
 \]
 Hence, the claim follows.
\end{proof}

%
%
%

\section{Eigenvalue statistics and the proof of Theorem~\ref{thm:main4}}
\label{sec:evstat}

Since we will focus on the case $k=2$, it will be convenient
to introduce the skew-shift $T:\T^2\to\T^2$ by
\be
 T(x,y) = (x + 2\omega, x + y) \pmod{1}.
\ee
One easily checks that this is equivalent
to \eqref{eq:defkskew} and that
\be
 T^n(x,y) = (x + 2 n \omega, y + n x+ n(n-1)\omega)\pmod{1}.
\ee
Then our Verblunsky coefficients are given by
\be
 \alpha_{x,y; n} = \lambda e(y + n x + n(n-1)\omega),
\ee
where we use the abbreviation $e(t) = \E^{2\pi\I t}$.

The main goal of this section is to prove the following
theorem, which will imply Theorem~\ref{thm:main4}.

\begin{theorem}\label{thm:evskew}
 Assume $\omega$ satisfies \eqref{eq:conddiop}.
 Let $x,y\in\T$ and $\beta,\gamma\in\partial\D$. 
 There exists $\sigma > 0$ such that for $N$
 sufficiently large, there exist 
 $\theta_1^N, \dots, \theta_N^N$ and $\vartheta^N$
 such that
 \be
  \sigma(\mathcal{E}_{x,y; \beta,\gamma}^{[0,N-1]})
   = \left\{\E^{2\pi\I \theta_1^N}, \dots, \E^{2\pi\I\theta_N^N}
     \right\}
 \ee
 and
 \be
  \frac{1}{N} \#\left\{n:\quad \|\theta_n^N - \vartheta^N + 2 n \omega\|
   > \frac{1}{N^{1 + \sigma}}\right\} \leq \frac{1}{N^{\sigma}}.
 \ee
\end{theorem}

In order to see how this implies Theorem~\ref{thm:main4},
we need to introduce some more notation related to the
Laplace functional. Given $N$ points $x_1^N, \dots, x_N^N \in \T$,
we define for $\theta\in\T$
\be
 \left[-\frac{1}{2},\frac{1}{2}\right) \ni x_n^N(\theta)
   = x_n^N - \theta \pmod{1}.
\ee
Then their Laplace functional is defined by
\be
 \mathfrak{L}_{x^N, N}(f) = \int_{\T} \exp\left(-\sum_{n=1}^{N}
   f(N x_n^N(\theta))\right) d\theta
\ee
where $f$ is a continuous, compactly supported, and positive
function. If $\ul{x} = \{ \{x_n^{N}\}_{n=1}^{N} \}_{N=1}^{\infty}$
is a sequence of vectors, we denote
\be
 \mathfrak{L}_{\ul{x}, N}(f) = \mathfrak{L}_{x^{N}, N}(f).
\ee
Theorem~\ref{thm:main4} follows by applying
(iv) of the next lemma to the sequences
\be
 \ul{\theta} = \{\{\theta_n^{N}\}_{n=1}^{N}\}_{N=1}^{\infty},\quad
 \ul{\vartheta} = \{\{\vartheta^{N} - 2 n\omega\}_{n=1}^{N}\}_{N=1}^{\infty}
\ee

\begin{lemma}
 Let $f:\R\to\R$ be a positive, continuous, and
 compactly supported function, 
 $\ul{x} = \{ \{x_n^{N}\}_{n=1}^{N}\}_{N=1}^{\infty}$
 and $\ul{y}$ be sequences of vectors in $\T$.
 \begin{enumerate}
  \item Let $c > 0$ and $A > 1$, then
   \be
    |\{\theta\in\T:\quad\#\{1\leq n\leq N:\quad N x_n^N(\theta)\in
     [-c,c]\} \geq A \}| \leq \frac{2c}{N}.
   \ee
  \item If $\max_{1\leq n\leq N} N \|x_n^{N} - y_n^{N}\| \to 0$
   then
   \be
    |\mathfrak{L}_{\ul{x},N}(f) - 
     \mathfrak{L}_{\ul{y},N}(f)| \to 0.
   \ee
  \item If
   \be
    \frac{1}{N} \#\{1\leq n \leq N:\quad x_n^{N} \neq y_n^{N}\} \to 0
   \ee
   then
   \be
    |\mathfrak{L}_{\ul{x},N}(f) - 
     \mathfrak{L}_{\ul{y},N}(f)| \to 0.
   \ee
  \item If for every $\eps > 0$
   \be
    \frac{1}{N} \#\{1\leq n \leq N:\quad \|x_n^{N} - y_n^{N}\| \geq \frac{\eps}{N}\} \to 0
   \ee
   then
   \be
    |\mathfrak{L}_{\ul{x},N}(f) - 
     \mathfrak{L}_{\ul{y},N}(f)| \to 0.
   \ee
 \end{enumerate}
\end{lemma}

\begin{proof}[Proof of (i)]
 Follows from 
 \[
  \int_{-\frac{1}{2}}^{\frac{1}{2}}
   \sum_{n=1}^{N} \chi_{\left[-\frac{c}{N}, \frac{c}{N}\right]}
   (x_n^{N}(\theta)) d\theta = 2 c.
 \]
 and Markov's inequality.
\end{proof}

\begin{proof}[Proof of (ii)]
 Let $\eps > 0$.
 Since $f$ is compactly supported, we have $\mathrm{supp}(f)
 \subseteq [-c,c]$. Let $A = \lceil \frac{2c}{10 \eps} \rceil$.
 By (i), there exists a set $I \subseteq \T$ such that
 for $\theta \in I$
 \[
  \#\{1\leq n \leq N:\quad N x_n(\theta) \in [-c,c]
   \text{ or } N y_n(\theta) \in [-c,c] \} \leq A
 \]
 and $|\T \setminus I| \leq \frac{\eps}{2}$.
 By assumption, $f$ is uniformly continuous, so there
 exists a $\delta > 0$ such that $|f(x) - f(y)| \leq \frac{\eps}{2 A}$
 for $|x-y| < \delta$. Choose $N$ so large that
 \[
  \max_{1\leq n \leq N} N \|x_n^{N} - y_n^{N}\| < \delta.
 \]
 Then we clearly have that $|N x_n(\theta) - N y_n(\theta)| < \delta$,
 and thus that for $\theta \in I$.
 \[
  \left|\sum_{n=0}^{N-1} f(N x_n(\theta))
   - \sum_{n=0}^{N-1} f(N y_n(\theta))\right| < \frac{\eps}{2}.
 \]
 The claim follows.
\end{proof}

\begin{proof}[Proof of (iii)]

 (ii) follows from the set of $\theta$ for which
 \[
  \sum_{n=1}^{N} f(N y_n^N(\theta)) \neq
  \sum_{n=1}^{N} f(N x_n^N(\theta))
 \]
 having vanishing measure as $N\to\infty$.
\end{proof}

\begin{proof}[Proof of (iv)]
 By assumption, there exists $\eps_N \to 0$ such that
 \[
  \frac{1}{N} \#\{1\leq n\leq N:\quad \|x_n^{N} - y_n^{N} \| 
   \geq \frac{\eps_N}{N} \} \to 0.
 \]
 Define
 \[
  \ti{x}_n^{N} = \begin{cases} y_n^{N},& \|x_n^{N} - y_n^{N}\| 
   \geq \frac{\eps_N}{N};\\
   x_n^N, & \text{otherwise}.\end{cases}
 \]
 Then $x^N$ and $\ti{x}^N$ satisfy the assumptions of (i)
 and $\ti{x}^N$ and $y^N$ the ones of (ii). The claim follows.
\end{proof}

We now begin the proof of Theorem~\ref{thm:evskew}.
\eqref{eq:conddiop} implies that there exists
some $c > 0$ such that
\be\label{eq:conddiop2}
 \|q \omega\| \geq \frac{c}{q^{\tau}}
\ee
for all positive integers $q$. The following
theorem will be essential to our proof and
proven only in the next section.

\begin{theorem}\label{thm:existtest}
 There is a constant $\sigma \in (0,1)$.
 Let $\eta \geq 1$, $N$ sufficiently large,
 $\beta,\gamma\in\partial\D$
 and $x,y\in\T$. There exists a normalized
 $\psi \in \ell^2(\{0,\dots,N-1\})$ such that
 $\psi(n) = 0$ for $n \geq N^{\sigma}$, $n =0,1$ 
 and $z = \E^{2\pi\I \vartheta}$ such that
 \be
  \|(\mathcal{E}^{[0,N-1]}_{x,y;\beta,\gamma} - z) \psi\|
   \leq \left(\frac{1}{N}\right)^{\eta}.
 \ee
\end{theorem}

Define $\vartheta_k = \vartheta - 2  \omega k$ and
$z_k = e(\vartheta_k)$.

\begin{lemma}
 Let $\eps > 0$.
 If \eqref{eq:conddiop2} holds, then for
 $N$ large enough and $k \neq \ti{k} \in \{0, \dots, N-1\}$
 we have
 \be
  |z_k - z_{\ti{k}}| \geq \frac{1}{N^{\tau + \eps}}.
 \ee
\end{lemma}

\begin{proof} Clear. \end{proof}

\begin{proof}[Proof of Theorem~\ref{thm:evskew}]
 We let $\eta = \tau + 2\eps$ in Theorem~\ref{thm:existtest}.
 With $u_{n}$ the appropriate factors as given
 in Lemma~\ref{lem:rotatealpha}, we define the test
 functions
 \[
  \psi_k(n) = \begin{cases} u_{n} \psi(n - k), &k\leq n\leq k+N^{\sigma}\\
   0, &\text{otherwise}. \end{cases}
 \]
 We then have for $0 \leq k \leq N - N^{\sigma}$ that
 \[
  \|(\mathcal{E}_{x,y;\beta,\gamma}^{[0,N-1]}-z_k) \psi_k\|
   \leq \frac{1}{N^{\eta}}.
 \]
 Hence, there is some eigenvalue $\E^{2\pi\I\theta_{\ell_k}}$
 such that
 \[
  \|\theta_{\ell_{k}} - \vartheta_k\| \leq \frac{1}{N^{\eta}}.
 \]
 By the previous lemma, we must have $\ell_k \neq \ell_{\ti{k}}$
 for $k \neq \ti{k}$. The claim then follows upon
 reordering the $\theta_{\ell}$.
\end{proof}

%
%
%

\section{Proof of Theorem~\ref{thm:existtest}}
\label{sec:existtest}

Let $L = \lfloor \frac{1}{3} N^{\sigma}\rfloor$.
If we show that for every $(x,y)\in\T^2$, there exists
a normalized vector $\psi$ and $z\in\partial\D$
such that
\be
 \|(\mathcal{E}^{[-L,L]}_{x,y; \beta,\gamma} - z) \psi\|
  \leq \frac{1}{N^C}
\ee
then Theorem~\ref{thm:existtest} follows. We will
show this modified claim, since it is notationally
somewhat simpler to deal with.

Since $\mathcal{E}^{[-L,L]}_{x,y; \beta,\gamma}$ has $2 L + 1$ eigenvalues,
there exists $z\in\D$ and $\|\psi\| = 1$ such that
\be
 \mathcal{E}^{[-L,L]}_{x,y; \beta,\gamma} \psi = z \psi,
  \quad |\psi(0)|^2 \geq \frac{1}{2L+1}.
\ee
We will prove in the following section

\begin{theorem}\label{thm:greendecays}
 There exists $\eta > 0$ such that for every $C \geq 1$,
 we have for $L$ large enough and $M = \lfloor L^{\eta} \rfloor$
 that there exist
 \be
  -\frac{2}{3} L \leq k_- \leq -\frac{1}{3} L,\quad
  \frac{1}{3} L \leq k_+ \leq \frac{2}{3} L
 \ee
 such that for 
 \be
  k \in \{k_- - C M, \dots, k_- + C M\}
   \cup \{k_+ - CM, \dots, k_+ + C M\}
 \ee
 we have that there
 exist $\beta, \gamma \in \partial\D$ such that
 for $|k - \ell| \leq \frac{M}{2}$ we have
 \be
  |G^{[k-M, k+M]}_{x,y;\beta,\gamma}(z; \ell, k-M)|,\ 
  |G^{[k-M, k+M]}_{x,y;\beta,\gamma}(z; \ell, k+M)| \leq \frac{1}{M}.
 \ee
\end{theorem}

Define
\be
 \mathcal{K}_{t} = 
  \{k_- - t M, \dots, k_- + t M\}
   \cup \{k_+ - t M, \dots, k_+ + t M\}.
\ee
Using Lemma~\ref{lem:green2sol} combined
with the estimate from the previous theorem, we
can conclude for $k \in K_{C}$ and $|\ell -k| \leq \frac{M}{2}$
that
\be
 |\psi(\ell)| \leq \frac{4}{M},
\ee
where we used the trivial estimate $|\psi(n)|\leq 1$

We can iterate this to obtain for $s = 1,\dots, C$ that
for $k \in K_{C-s+1}$ and $|\ell -k| \leq \frac{M}{2}$
\be
 |\psi(\ell)| \leq \left(\frac{4}{M}\right)^s.
\ee
In particular, we obtain that
\be
 |\psi(k_-)|, |\psi(k_+)| \leq \left(\frac{4}{M}\right)^C.
\ee
Define a test function $\varphi$ by
\be
 \varphi(n) = \begin{cases} \psi(n), & k_- \leq n \leq k_+; \\
  0,&\text{otherwise}.\end{cases}
\ee
We have that
\be
 \|(\mathcal{E}^{[-L,L]}_{x,y;\beta,\gamma} - z) \varphi\| 
  \leq  \left(\frac{8}{L}\right)^{C \eta}
\ee
and thus Theorem~\ref{thm:existtest} follows.

%
%
%

\section{Decay of the Green's function: 
  Proof of Theorem~\ref{thm:greendecays}}
\label{sec:greendecays}

Theorem~\ref{thm:uegreen} states that
$m = L(z) > 0$ implies that for $N$ large enough
there exists a set $B_N \subseteq\T^2$ with
\begin{enumerate}
 \item $|B_N|\to 0$.
 \item For $(x,y) \in \T^2\setminus B_N$, there exists
  \be
   \beta \in \left\{-\frac{\alpha_{-N-1}}{|\alpha_{-N-1}|},
    \frac{\alpha_{-N-1}}{|\alpha_{-N-1}|}\right\},\quad
   \gamma \in \left\{\frac{\alpha_N}{|\alpha_N|},
    -\frac{\alpha_N}{|\alpha_N|}\right\}
  \ee
  such that for $|k| \leq \frac{N}{2}$, we have
  \be
   |G^{[-N,N]}_{x,y;\beta,\gamma}(z; -N, k)| 
   \leq \E^{-\frac{m}{4} |k+N|},
   \quad
   |G^{[-N,N]}_{x,y;\beta,\gamma}(z; N, k)| 
   \leq \E^{-\frac{m}{4} |N-k|}.
  \ee
\end{enumerate}

We will first need the following lemma.

\begin{lemma}
 Let $\sigma > 0$, there exists a constant $C > 1$ such
 that for $N \geq 1$, there exists a set $B^2_N$ such
 that
 \be
  |B_N^2| \leq \frac{C}{N^{\sigma}}
 \ee
 and for
 \be
  \beta \in \left\{-\frac{\alpha_{-N-1}}{|\alpha_{-N-1}|},
   \frac{\alpha_{-N-1}}{|\alpha_{-N-1}|}\right\},\quad
  \gamma \in \left\{\frac{\alpha_N}{|\alpha_N|},
   -\frac{\alpha_N}{|\alpha_N|}\right\}
 \ee
 and $x,y\in\T^2\setminus B_N^{2}$ we have
 \be
  \left\|\left(z \left(
   \mathcal{L}^{[-N,N]}_{x,y;\beta,\gamma}\right)^{*} - 
    \mathcal{M}^{[-N,N]}_{x,y;\beta,\gamma}\right)^{-1}
     \right\|
      \leq N^{1 + \sigma}
 \ee
\end{lemma}

\begin{proof}
 This is a consequence of Theorem~\ref{thm:wegneresti}.
\end{proof}

In summary, we have extracted the following statement

\begin{proposition}
 Let $\sigma > 0$. For $N \geq 1$ large enough, there
 exists $\Omega_N \subseteq\T^2$ such that
 \be
  \lim_{N\to\infty} |\T^2 \setminus\Omega_N| = 0.
 \ee
 For each $(x,y) \in \Omega_N$ and
 \be
  |\ti{x} - x| \leq \frac{1}{N^{2(1+2\sigma)}},\quad
  |\ti{y} - y| \leq \frac{1}{N^{1+2\sigma}},
 \ee
 we have that there exists
 \be
  \beta \in \left\{-\frac{\alpha_{-N-1}}{|\alpha_{-N-1}|},
   \frac{\alpha_{-N-1}}{|\alpha_{-N-1}|}\right\},\quad
  \gamma \in \left\{\frac{\alpha_N}{|\alpha_N|},
   -\frac{\alpha_N}{|\alpha_N|}\right\}
 \ee
 such that for $|k| \leq \frac{N}{2}$, we have
 \be
  |G^{[-N,N]}_{\ti x, \ti y;\beta,\gamma}(z; -N, k)| \leq \frac{1}{N},
  \quad
  |G^{[-N,N]}_{\ti x, \ti y;\beta,\gamma}(z; N, k)| \leq \frac{1}{N}.
 \ee
\end{proposition}

\begin{proof}
 A computation shows that
 \[
  \|\mathcal{L}^{[-N,N]}_{x,y;\beta,\gamma} - 
   \mathcal{L}^{[-N,N]}_{\ti{x},\ti{y};\beta,\gamma}\|
   \lesssim N |x - \ti{x}| + |y + \ti{y}| 
    \leq \frac{1}{N^{\frac{\sigma}{2}}}
 \]
 for $N$ large enough
 and a similar result for
  $\mathcal{M}^{[-N,N]}_{\ti{x},\ti{y};\beta,\gamma}$. The result now
 follows from
 \[
  B^{-1} - A^{-1} = B^{-1} (A - B) A^{-1}
 \]
 and some computations.
\end{proof}

Let $X_N = \lceil N^{2 (1 + 2\sigma)} \rceil$,
$Y_N = \lceil N^{1+2\sigma} \rceil$. We partition
$\T^2$ into $X_N \cdot Y_N \lesssim N^{3(1+2\sigma)}$
boxes of side length $\frac{1}{X_N}$ and $\frac{1}{Y_N}$.
We call a box $I_{\ell}$ bad if
\be
 \Omega_N \cap I_{\ell} = \emptyset
\ee
and good otherwise. We note that if $(x,y)$ is in a good
box, then for $|k|\leq\frac{N}{2}$
\be
 |G^{[-N,N]}_{x,y; \beta, \gamma}(z; k, \pm N)| \leq \frac{1}{N}
\ee
for some $\beta, \gamma\in\partial\D$.
We now given an upper bound on the number of iterates
of $T^{j}(x,y)$ that land in any bad box. 
We will show the following theorem in Appendix~\ref{sec:dynskew}.
For $\eps, \delta >0$, denote by $B_{\eps,\delta}
\subseteq\T^2$
the set
\be
 B_{\eps,\delta} = \{(x,y)\in\T^2:\quad
  \|x\|\leq \eps,\ \|y\|\leq\delta\}.
\ee

\begin{theorem}\label{thm:skewergodic2}
 Assume \eqref{eq:conddiop} and let
 $\delta > 0$, $\eps > 0$, $N \geq 1$. There exists $L_0 = L_0(\sigma,\omega) \geq 1$
 such that for any $x,y\in\T$ there exists $0 \leq \ell_0 \leq N$
 such that for $L \geq L_0 \delta^{-4} \eps^{-9}$
 \be
  \#\{0 \leq \ell \leq \frac{L_0}{N}:\quad
   T^{\ell N}(x,y) \in B_{\eps, \delta}\} \leq 10 \eps \delta \frac{L}{N}.
 \ee
\end{theorem}

We now obtain that for $L \geq N^{15}$ and $N$ large enough, we have
for some $0 \leq \ell_0 \leq N-1$
\be
 \#\{\lfloor\frac{1}{3N } L+\ell_0\rfloor \leq \ell \leq \frac{2}{3N} L:
   \quad T^{\ell N}(x,y)\text{ in fixed
  bad box}\} \leq \frac{10 L/N}{N^{3 (1 + \sigma)}}.
\ee
Since
\be
 \#\{\text{bad boxes}\} \leq \delta_N N^{3(1+\sigma)}
\ee
with $\delta_N \to 0$ as $N \to \infty$, we obtain for $L \geq N^{15}$ that
\be
 \#\{\lfloor\frac{1}{3N } L+\ell_0\rfloor \leq \ell \leq \frac{2}{3N} L:
  \quad T^{\ell N}(x,y)\text{ in some bad box}\} \leq \delta_N \frac{L}{N}
\ee
for $\delta_N\to 0$ as $N \to \infty$.

\begin{proof}[Proof of Theorem~\ref{thm:greendecays}]
 We just give the argument for $k_+$.
 Choose $\delta_N \leq \frac{1}{10 C}$.
 Now divide $\left[\lfloor\frac{1}{3N } L+\ell_0\rfloor,
 \frac{2}{3N} L\right]$ into segments of length $3 C$.
 Then at most $\delta_N \frac{L}{N}$ of them can
 contain an iterate that lands in a bad box,
 but there are $\frac{1}{C} \frac{L}{3 N} = 10 \delta_N
 \frac{L}{3N}$ many of them. Hence, we must have at
 least one, where our conclusion holds.
\end{proof}

\appendix

%
%
%

\section{Dynamics of the skew-shift}
\label{sec:dynskew}

In this section, we will discuss quantitative recurrence
results for the skew-shift. The discussing here follows
the one in Chapters~10 and 11 in \cite{kthesis}.

\begin{theorem}\label{thm:skewergodic}
 Assume \eqref{eq:conddiop2} and let
 $\sigma > 0$. Then for a constant $C = C(c,\tau,\sigma ) >0$
 we have for $L\geq 1$
 \be
  \#\{1\leq\ell\leq L:\quad
  T^{\ell}(x,y) \in B_{\eps,\delta}\} \leq 5 \eps \delta L
   + C \left(\frac{1}{\eps}\right)^{1 + \sigma}
   L^{\frac{1}{2} + \sigma}.
 \ee
\end{theorem}

\begin{proof}[Proof of Theorem~\ref{thm:skewergodic2}]
 Let $\sigma = \frac{1}{4}$ in the previous theorem.
 Then we have that
 \[
  \#\{1\leq\ell\leq L:\quad
  T^{\ell}(x,y) \in B_{\eps,\delta}\} \leq 10 \eps \delta L
 \]
 if $ C \left(\frac{1}{\eps}\right)^{\frac{5}{4}}
   L^{\frac{3}{4}} \leq 5 \eps \delta L$
 or equivalently
 \[
  L^{\frac{1}{4}} \geq \frac{C}{5 \delta} \cdot \frac{1}{\eps^{\frac{9}{4}}}.
 \]
 Next divide $1 \leq \ell \leq L$ into $N$ arithmetic
 progressions of the form $\{\ell_0 + \ell N\}_{\ell=0}^{L/N}$
 for $\ell_0 \in \{1,\dots, L\}$. Then at least
 one of them must contain less than $\frac{10 \eps \delta L}{N}$
 elements.
\end{proof}

We now begin to prove Theorem~\ref{thm:skewergodic}. 

\begin{lemma}
 There exists a trigonometric polynomial $P$ given
 by
 \be
  P(x,y) = \sum_{|j| \leq \frac{2}{\eps}}
   \sum_{|k| \leq \frac{2}{\delta}} P_{j,k} e(jx+k y)
 \ee
 such that $|P_{j,k}|\leq 5\eps\delta$ and
 \be
  \chi_{B_{\eps,\delta}} \leq P,
 \ee
 where $\chi_{A}$ denotes the characteristic function
 of $A\subseteq\T^2$.
\end{lemma}

\begin{proof}
 Follows by using Selberg polynomials,
 see Chapter~2 in \cite{mont}.
\end{proof}

We compute that
\begin{align*}
  \#\{1\leq\ell\leq L:\quad
  &T^{\ell}(x,y) \in B_{\eps,\delta}\} \leq \sum_{\ell=1}^{L}
  P(T^{\ell}(x,y)) \\
 & \leq 5 \eps \delta L + 5 \eps \delta
   \sum_{|j| \leq \frac{2}{\eps}}
   \sum_{|k| \leq \frac{2}{\delta}}
   \left|\sum_{\ell=1}^{L} e(j \cdot 2 \ell \omega + k x \ell
    - k \omega \ell
    + k \omega \ell^2 )\right|.
\end{align*}
To finish the proof of Theorem~\ref{thm:skewergodic}
we will need the next two bounds.

\begin{lemma}
 We have
 \be
  \left|\sum_{\ell=1}^{L} e(\ell \cdot \omega)\right|
   \leq \frac{1}{2 \|\omega\|}.
 \ee
 and for $\sigma > 0$, there exists $C = C(\sigma) > 0$
 such that for any $t\in\R$
 \be
  \left|\sum_{\ell=1}^{L} e(t \ell + \omega \ell^2)\right|
   \leq \frac{C L^{\frac{1}{2} + \sigma}}{\|\omega\|}.
 \ee
\end{lemma}

\begin{proof}
 See Chapter~3 in \cite{mont}.
\end{proof}

\begin{proof}[Proof of Theorem~\ref{thm:skewergodic}]
 From the previous lemma and the computation preceeding
 it, we obtain
 \[
 \#\{1\leq\ell\leq L:\quad
  T^{\ell}(x,y) \in B_{\eps,\delta}\} \leq 5 \eps \delta L
  + C L^{\frac{1}{2} + \sigma} 
   \sup_{1 \leq k \leq \frac{2}{\eps}} \frac{1}{\|k \omega\|}.
 \]
 The claim now follows by \eqref{eq:conddiop}.
\end{proof}

%
%
%

\section*{Acknowledgements}

The key realization that Proposition~\ref{prop:rotatealphas}
and Lemma~\ref{lem:rotatealpha} hold, came to me 
during discussions with Darren Ong.
Furthermore, I am thankful to Maxim Zinchenko for
correspondence, which clarified issues related
to \eqref{eq:varphibulletgamma}.

%
%
%

\end{document}